\documentclass[12pt]{article}

\usepackage{amsmath}
\usepackage{amsthm}
\usepackage{amsfonts}
\usepackage{mathrsfs}
\usepackage{stmaryrd}
\usepackage{setspace}
\usepackage{fullpage}
\usepackage{amssymb}
\usepackage{breqn}
\usepackage{enumitem}
\usepackage{bbold}
\usepackage{authblk}
\usepackage{comment}
\usepackage{hyperref}
\usepackage{pgf,tikz}
\usepackage{graphicx}
\usepackage{subcaption}

\newtheorem*{thm*}{Theorem}
\newtheorem{thm}{Theorem}[section]

\newtheorem{lemma}[thm]{Lemma}
\newtheorem{conjecture}[thm]{Conjecture}

\newtheorem*{prop*}{Proposition}

\newtheorem{corollary}[thm]{Corollary}
\newtheorem{question}[thm]{Question}

\newtheorem{claim}[thm]{Claim}

\newcommand{\nchn}[1]{\binom{#1}{\lfloor \frac{#1}{2}\rfloor}}

\newcommand{\ignore}[1]{}

\mathcode`l="8000
\begingroup
\makeatletter
\lccode`\~=`\l
\DeclareMathSymbol{\lsb@l}{\mathalpha}{letters}{`l}
\lowercase{\gdef~{\ifnum\the\mathgroup=\m@ne \ell \else \lsb@l \fi}}
\endgroup

\title{Query complexity of Boolean functions\\ on the middle slice of the cube}
\author{D\'aniel Gerbner$^{a}$, Bal\'azs Keszegh$^{a,b}$,
D\'aniel T. Nagy$^{a}$, Kartal Nagy$^{b}$, D\"om\"ot\"or P\'alv\"olgyi$^{b,a}$, Bal\'azs Patk\'os$^a$, G\'abor Wiener$^c$
\\
\small $^a$ HUN-REN Alfr\'ed R\'enyi Institute of Mathematics\\
\small $^b$ ELTE E\"otv\"os Lor\'and University\\
\small $^c$ Budapest University of Technology and Economics, \\
\small Department of Computer Science and Information Theory \\
}
\date{}

\begin{document}

\maketitle

\begin{abstract}
     We study the query complexity on slices of Boolean functions. Among other results we show that there exists a Boolean function for which we need to query all but 7 input bits to compute its value, even if we know beforehand that the number of 0's and 1's in the input are the same, i.e., when our input is from the middle slice.
     This answers a question of Byramji.
     Our proof is non-constructive, but we also propose a concrete candidate function that might have the above property.
     Our results are related to certain natural discrepancy type questions that, somewhat surprisingly, have not been studied before.
\end{abstract}

\section{Introduction}

In this note we consider 0--1 functions on the set of $k$-element subsets of an $n$-element set. We use standard notation: $[n]=\{1,2\dots,n\}$, $[a,b]=\{a,a+1,\dots,b\}$, $2^S=\{T:T\subseteq S\}$, $\binom{S}{k}=\{T\subseteq S:|T|=k\}$. Using these, our main interest is the \textit{query complexity} of Boolean functions $f\colon \binom{[n]}{k}\rightarrow \{0,1\}$. The deterministic query
complexity $D_k(f)$ of a function $f$ is the minimum number of queries that a deterministic
adaptive algorithm needs to compute $f(A)$ correctly for every $A\in \binom{[n]}{k}$  with queries being elements $i$ of $[n]$ and the answer being $i\in A$ or $i \notin A$.

This can be viewed as a two-player game between the Questioner and the Adversary. In the game, $f$ is known to both players and the Adversary thinks of a set $A\in \binom{[n]}{k}$. In a round, Questioner may ask whether $i\in A$ for some $i\in [n]$, and then the Adversary answers this. Questioner's aim is to determine $f(A)$ with using as few queries as possible, while the Adversary tries to postpone this by possibly ``changing $A$ in his mind" such that it stays consistent with his previous answers. The number of rounds in a game played according to optimal strategies by both players is the query complexity of $f$, denoted by $D_k(f)$.
Let
$D_k(n)=\max \{D_k(f)\mid f:\binom{[n]}{k}\rightarrow \{0,1\}\}$ denote the maximum worst case deterministic query complexity over all Boolean functions on the slice $\binom{[n]}{k}$.
As we will be mainly interested in how close $D_k(n)$ can get to $n$, we introduce $E_k(n):=n-D_k(n)$.

Our starting points are the following results of Byramji. For a more detailed introduction on query complexity of Boolean functions on different domains, we refer the interested reader to his manuscript \cite{byramji}. In \cite{ow}, complexity problems on slices were investigated to prove a robust version of the Kruskal--Katona theorem and to find an optimal weak-learning algorithm for monotone functions.

\begin{thm}[Byramji \cite{byramji}]\label{fixthm}\
\begin{enumerate}
    \item $E_2(n)=\Theta(\log n)$.
    \item
    For any fixed $k$, $\Omega(\log^{(\binom k2)}) n=E_k(n)= O(\sqrt{\log^{(k-2)} n})$, where $\log^{(k)} n$ denotes the $k$ times iterated logarithm.
\end{enumerate}

\end{thm}

\begin{thm}[Byramji \cite{byramji}]\label{felthm}
$E_{n/2}(n)= O(\log\log n)$. 
\\
Furthermore, there exists an explicit function $f:\binom{[n]}{n/2}\rightarrow \{0,1\}$ with $E_{n/2}(f)=O(\log n)$.
\end{thm}

The main result of this note is the following improvement on Theorem \ref{felthm}.

\begin{thm}\label{main} \
\begin{enumerate}
    \item[(i)]
    $E_{n/2}(n)\le 7$
    for all $n$.
\item[(ii)]
$E_{n/2}(n)\le 5$
for all $n\ge 100$.
\item[(iii)]
For any constant $\alpha\le 1/2$ there exists $C=C_\alpha$ such that $E_{\alpha n}(n)\le C$.
\end{enumerate}
\end{thm}

We also prove some ``monotonicity results" which weakly answer another problem raised by Byramji. If the Adversary gives away the information $n\in A$ without the Questioner asking the query $n$, then the Questioner needs to determine $f'(B)$ for $f':\binom{[n-1]}{k-1}\rightarrow \{0,1\}$ with $f'(B):=f(A\cup \{n\})$. This implies
$D_k(n)\ge D_{k-1}(n-1)$. Similarly, if
the Adversary gives away the information
$n\notin A$ for free, then we obtain
$D_k(n)\ge D_k(n-1)$.
These inequalities are equivalent to $E_k(n)\le E_{k-1}(n-1)+1$
and
$E_k(n)\le E_k(n-1)+1$.
Byramji \cite{byramji} asked whether
$E_{k'}(n)\le E_{k}(n)$
holds for all $k\le k'\le n/2$.
We do not have a proof for this, but we can show the following slightly weaker statement.

\begin{lemma}\label{lem:order}
    For every $n=n_1+n_2$ and $k=k_1+k_2$, we have $E_k(n)\le E_{k_1}(n_1)+E_{k_2}(n_2)$.
\end{lemma}

By choosing $n_2=2k_2$, we get the following monotonicity result from Theorem \ref{main} (i).

\begin{corollary}\label{cor:order}
    For any $k_1\le k\le n/2$, we have $E_{k}(n)\le E_{k_1}(n-2(k-k_1))+7$.
\end{corollary}

In particular, since by Theorem \ref{fixthm}, we have $E_c(n)= O(\sqrt{\log^{(c-2)} n})$ for any constant $c$, this implies that $E_k(n)= O(\sqrt{\log^{(c-2)} n})$ holds for any $k\ge c$.
While we could prove $E_k(n)= O(1)$ only for $k=\Omega(n)$, we conjecture that this holds already for much smaller $k$.

\section{Proofs}

The proof of Theorem \ref{main} is a counting proof and relies heavily on the notion of decision trees. A decision tree is a rooted  binary tree. Vertices represent the queries of the strategy, so, as we can assume that a strategy is wise enough not to ask the same query twice, an inner node at level $n-\ell$ can be labelled $\ell$ ways, depending on the index $i$ of the yet unqueried locations at the node. Edges are labelled with 0 or 1 based on the answers to the query of their parent node with 1 meaning $i \in A$ and 0 meaning $i \notin A$. Finally, leaves are labelled with 0 or 1, as if the strategy does not continue, then we should know the value of the function.  We bound the number of strategies of length at most $n-t$ with the following lemma.

\begin{lemma}\label{count}
    The total number of decision trees of height at most $n-t$ is at most
 \[
g(n,k,t):=2^{\sum_{i=k-t}^{k}\binom{n-t}{i}}
\cdot  \prod_{\ell=t+1}^n \ell^{\sum_{i=k-\ell}^{k}\binom{n-\ell}{i}-\binom{n-l-1}{k-1}-\binom{n-l-1}{n-k+1}}\cdot 2^{\binom{n-l-1}{k-1}+\binom{n-l-1}{n-k+1}}.
\]
\end{lemma}

\begin{proof}
Observe that we do not have to consider complete binary trees, as if $k$ 1-answers or $n-k$ 0-answers (and thus $n-\ell-(n-k)=k-\ell$ 1-answers) are given, then we know the input, so we know the value of the function.
Therefore, on level $n-\ell$, there are at most
$\sum_{i=k-\ell}^{k}\binom{n-\ell}{i}$
vertices,
of which
$\binom{n-l-1}{k-1}+\binom{n-l-1}{n-k+1}$ are leaves.
Since inner vertices can be labelled in $\ell$ ways, and leaves in $2$ ways, this gives at most
\[\ell^{\sum_{i=k-\ell}^{k}\binom{n-\ell}{i}-\binom{n-l-1}{k-1}-\binom{n-l-1}{n-k+1}}\cdot 2^{\binom{n-l-1}{k-1}+\binom{n-l-1}{n-k+1}}
\]
options for level $n-l$ when $l>t$, and
\[2^{\sum_{i=k-t}^{k}\binom{n-t}{i}}
\]
for level $n-t$ which can only have leaves.
    This proves the lemma.
\end{proof}

\begin{lemma}\label{kozep}
  For the function  \[
f(n,t):=\frac{(t+1)\nchn{n-t}+\sum_{\ell=t+1}^{n}
 l\log l\nchn{n-\ell}+2\nchn{n-\ell}}{\nchn{n}}.
\]
we have $f(n,7)<1$ for all $n\ge 7$ and $f(n,5)<1$ for all $n\ge 100$.
\end{lemma}

\begin{proof}
    The inequality $f(n,7)<1$ can be checked easily by computer for $7\le n <100$. Since $f(n,7)<f(n,5)$, it is enough to prove that $f(n,5)<1$ if $n\ge 100$.

Now we give an upper bound on $f(n,5)$ assuming $n\ge 100$. We will use that if $n'$ is odd and $n'<n$, then
$$\frac{\nchn{n-\ell}}{\nchn{n}}\le\frac{\nchn{n'-\ell}}{\nchn{n'}}.$$

We bound the tail end of the sum containing very small terms by a geometric series of quotient $\frac{2}{3}$ (in real, the ratios tend to $\frac12$).

$$f(n,5)=6\frac{\nchn{n-5}}{\nchn{n}}+\sum_{\ell=6}^{15}
 (l\log l+2)\frac{\nchn{n-\ell}}{\nchn{n}}+\sum_{\ell=16}^{n}
 (l\log l+2)\frac{\nchn{n-\ell}}{\nchn{n}}<$$

$$6\frac{\binom{94}{47}}{\binom{99}{49}}+\sum_{\ell=6}^{15}
 (l\log l+2)\frac{\nchn{99-\ell}}{\binom{99}{49}}+(16\log(16)+2)\frac{\nchn{n-16}}{\nchn{n}}\sum_{i=0}^{\infty} \left(\frac{2}{3}\right)^i<$$

 $$0.2+0.71+66\cdot \frac{\binom{83}{41}}{\binom{99}{49}}\cdot 3<0.92.$$
 This finishes the proof the lemma.
\end{proof}

\begin{lemma}\label{alfa}
    For any $0<\alpha<1/2$ there exists a constant $C=C_\alpha$ such that
    $$\sum_{l=C}^n \log(l) \sum_{i=k-\ell}^{k}\binom{n-\ell}{i} < \binom{n}{k}$$
    holds if $n$ is large enough and $k=\lfloor\alpha\cdot n\rfloor$.
\end{lemma}

\begin{proof}
    We reverse the order of summation and split the resulting sum in two parts:
$$\sum_{l=C}^n \log(l) \sum_{i=k-\ell}^{k}\binom{n-\ell}{i} < \sum_{i=0}^k \sum_{l=C}^n \log(l)\binom{n-\ell}{i}=$$
$$\sum_{i=0}^{\lfloor k/2 \rfloor} \sum_{l=C}^n \log(l)\binom{n-\ell}{i} + \sum_{i=\lfloor k/2 \rfloor+1}^k \sum_{l=C}^n \log(l)\binom{n-\ell}{i} $$

We estimate the first part as follows:
$$\sum_{i=0}^{\lfloor k/2 \rfloor} \sum_{l=C}^n \log(l)\binom{n-\ell}{i}<n^2\log(n)\binom{n-C}{\lfloor k/2 \rfloor}$$

This is less than $\frac{1}{2}\binom{n}{k}$ if $n$ is large enough.

To estimate the second part, observe that for $l>C$ and $i>k/2$

$$\frac{\log(l+1)\binom{n-\ell-1}{i}}{\log(l)\binom{n-\ell}{i}}=\frac{\log(l+1)}{\log(l)}\cdot\frac{n-l-i}{n-l}<\frac{\log(C+1)}{\log(C)}\cdot\frac{n-k/2}{n}=\frac{\log(C+1)}{\log(C)}(1-\alpha/2),$$
which is smaller than $1-\alpha/4$ for sufficiently large $C=C_\alpha$. Therefore we can upper bound the second part with a geometric series of quotient $1-\alpha/4$.

$$\sum_{i=\lfloor k/2 \rfloor+1}^k \sum_{l=C}^n \log(l)\binom{n-\ell}{i}\le \sum_{i=\lfloor k/2 \rfloor+1}^k \frac{4}{\alpha}\log(C)\binom{n-C}{i}$$

For $i\le k$, large $n$ and constant $C$ and we have
$$\frac{\binom{n-C}{i-1}}{\binom{n-C}{i}}=\frac{i}{n-C-i+1}<\frac{k}{n-k-C+1}<\beta<1$$
for some $\beta$ depending on $\alpha$. We can use a geometric series to get an upper bound here as well:

$$\frac{4}{\alpha}\log(C)\sum_{i=\lfloor k/2 \rfloor+1}^k \binom{n-C}{i}<\frac{4}{\alpha}\log(C)\frac{1}{1-\beta}\binom{n-C}{k}<$$
$$\frac{4}{\alpha}\log(C)\frac{1}{1-\beta}\binom{n}{k}\left(\frac{n-k}{n}\right)^C=\frac{4}{\alpha}\log(C)\frac{1}{1-\beta}(1-\alpha)^C\binom{n}{k}$$
This is less than $\frac{1}{2}\binom{n}{k}$ if $C=C_\alpha$ is a large enough constant, finishing the proof of the lemma.
\end{proof}

\begin{proof}[Proof of Theorem \ref{main}]

Observe that if $g(n,k,t)<2^{\binom{n}{k}}$ holds, then, by Lemma \ref{count} and as the number of functions $\binom{[n]}{k} \rightarrow \{0,1\}$ is $2^{\binom{n}{k}}$,  there exists $f$ that cannot have a decision tree of height at most $n-t$, and thus $E_n(k)\le t$.

To prove (i) and (ii), we observe \[\ell^{\sum_{i=k-\ell}^{k}\binom{n-\ell}{i}-\binom{n-l-1}{k-1}-\binom{n-l-1}{n-k+1}}\cdot 2^{\binom{n-l-1}{k-1}+\binom{n-l-1}{n-k+1}}
\le 2^{l\log l\nchn{n-\ell}+2\nchn{n-\ell}}
\]
and
\[2^{\sum_{i=k-t}^{k}\binom{n-t}{i}}\le 2^{(t+1)\nchn{n-t}}.
\]
Therefore, we obtain
\[
g(n,\lfloor n/2\rfloor,t)\le 2^{(t+1)\nchn{n-t}}\cdot \prod_{\ell=t+1}^{n}
 2^{l\log l\nchn{n-\ell}+2\nchn{n-\ell}}.
\]
We need this to be less than $2^{\nchn{n}}$.
By taking logarithm and dividing by $\nchn{n}$, the required inequality becomes
\[
\frac{(t+1)\nchn{n-t}+\sum_{\ell=t+1}^{n}
 l\log l\nchn{n-\ell}+2\nchn{n-\ell}}{\nchn{n}}<1,
\]
which is exactly the statement of Lemma \ref{kozep}.

To prove (iii), we upper bound all terms of $g(n,k,C)$ by $l^{ \sum_{i=k-\ell}^{k}\binom{n-\ell}{i}} $ and obtain
$$g(n,k,t)\le\prod_{l=C}^n l^{ \sum_{i=k-\ell}^{k}\binom{n-\ell}{i}} < 2^{\binom{n}{k}}  $$
The last inequality is equivalent to
$$\sum_{l=C}^n \log(l) \sum_{i=k-\ell}^{k}\binom{n-\ell}{i} < \binom{n}{k},$$
which holds by Lemma \ref{alfa}.
\end{proof}

\begin{proof}[Proof of Lemma \ref{lem:order}]
    For any $f_1:\binom{[n_1]}{k_1}\rightarrow \{0,1\}$ and $f_2:\binom{[n_1+1,n]}{k_2}\rightarrow \{0,1\}$,
    define $f:\binom{[n]}{k}\rightarrow \{0,1\}$ as follows:
    $f(A)=1$ if and only if $|A \cap [n_1]|=k_1$ and $f_1(A\cap [n_1])=f_2(A\cap [n_1+1,n])$.
    The Adversary can let the Questioner know that he is thinking of a set $A$ with $|A \cap [n_1]|=k_1$.
    Then to reveal $f(A)$, the Questioner needs to know both $f_1(A\cap [n_1])$ and $f_2(A\cap [n_1+1,n])$, so $D_{k}(n)\ge D_{k_1}(f_1)+D_{k_2}(f_2)$, which proves the statement if we chose some $f_1$ and $f_2$ that maximize the right-hand side.
\end{proof}

\section{Related discrepancy games}

We tried to find explicit (sequences of) functions $f:\binom{[n]}{n/2}\rightarrow \{0,1\}$ with $E_{n/2}(f)=O(1)$, i.e., $D_{n/2}(f)\ge n-C$ for some constant $C$.
Functions related to discrepancy seemed to be good candidates.
There exists a literature of positional discrepancy games, see \cite{alonetal,friezeetal,hefetal,szekely} and the monograph \cite{beck}, but the games that fit our purposes seem not to have been covered.
Let us mention an explicit function that we studied in detail.
\smallskip

$\textit{Disc-max-}d(A)=1$ if and only if $|A\cap [j]-j/2|\le d/2$ for all $1\le j\le n$.
\smallskip

To study this function, it is more comfortable in this section to use the following notation.
Set $x_i^t=0$ if it has not been queried yet in the first $t$ rounds, $x_i^t=+1$ if $i$ was queried and $i\in A$, and $x_i^t=-1$ if $i$ was queried and $i\notin A$.
Define the \emph{signed} discrepancy of a prefix as $disc^t(j)=\sum_{i=1}^j x_i^t$ and $disc(j)=disc^n(j)$.
With this notation, and renaming the image from $\{0,1\}$ to false and true:
\smallskip

$\textit{Disc-max-}d(A)$ is true if and only if $|disc(j)|\le d$ for all $1\le j\le n$.

\begin{conjecture}
    $E_{n/2}(\textit{Disc-max-}d)=O(1)$ for a suitable constant $d$.
\end{conjecture}

The above conjecture might hold already for $d=10$.

While we cannot prove the conjecture, we can give some sufficient conditions for the game not to be over yet at some point $t$.

\begin{claim}\label{claim:discmax}\
    \begin{enumerate}
        \item[(i)] For any $d>0$, if $|disc^t(j)|\le \frac d2$ for every $j$ at some point $t$, then $\textit{Disc-max-}d(x)$ being true is possible.
        \item[(ii)] For any $d$, if $|disc^t(j)|\le d$ for every $j$ at some point $t$, and there are yet $3d+1$ unqueried positions, then $\textit{Disc-max-}d(x)$ being false is possible.
    \end{enumerate}
\end{claim}
\begin{proof}
    (i) So far we have assigned at most $\lfloor \frac d2\rfloor$ more $-1$ values than $+1$ values, as otherwise the discrepancy of the whole interval would not be at most $\frac d2$, and vice versa.
    Without loss of generality, assume that so far we have assigned more $-1$ values.
    Assign the remaining values in an alternating manner, putting the at most $\lfloor \frac d2\rfloor$ `extra' $+1$ values to the end.
    This way the discrepancy is always at most $\max(\lfloor \frac d2\rfloor+1;\lfloor \frac d2\rfloor+\lfloor \frac d2\rfloor)\le d$.

    \textit{Remark.} This is best possible, as shown by the state where $x_i^t=-1$ if $i\le \lfloor \frac d2\rfloor+1$, $x_i^t=0$ if $\lfloor \frac d2\rfloor+2\le i\le d+1$, $x_i^t=+1$ if $d+2\le i\le 2d+2$ and $x_i^t=(-1)^i$ if $2d+2<i$.

    (ii) So far we have assigned at most $d$ more $-1$ values than $+1$ values, as otherwise the discrepancy of the whole interval would not be at most $d$, and vice versa.
    In particular, there are at least $d+1$ more $-1$'s and at least $d+1$ more $+1$'s we can assign.
    Let $j$ be the index of the $(d+1)$-st hitherto unqueried position.
    Without loss of generality, assume $disc^t(j)\ge 0$.
    If we assign $+1$'s to the first $d+1$ unqueried positions, then $|disc(j)|>d$, so $\textit{Disc-max-}d(x)$ is false.
\end{proof}

Now, we introduce a two-player game to bound $E_{n/2}(\textit{Disc-max-}d)$; this game is more similar to typically studied discrepancy questions.
As usual, the board is an array $x$ of length $n$ that has only 0 entries at the beginning of the game, i.e., $x^0_i=0$ for all $i$. In the $t$-th round, a player called \textit{Positioner} asks an index $i$ with $x^{t-1}_i=0$ for some $1\le i\le n$, and then a player called \textit{Signgiver}, as an answer, sets $x^t_i=+1$ or $x^t_i=-1$.
The values of the other indices remain unchanged, i.e., $x^t_j=x^{t-1}_j$ for $j\ne i$.
The goal of Positioner is to make the discrepancy $|disc^t(j)|$ big over some prefix $[j]$ at some time $t$, while the goal of Signgiver is to keep it low.
Denote the value of this game by $d(n)$, that is, $\max_t \max_j |disc^t(j)|$ during an optimal run of this game.
A simple corollary of Claim \ref{claim:discmax} is the following.

\begin{corollary}
    For $d=2d(n)$, we have $E_{n/2}(\textit{Disc-max-}d)\le 3d$.
\end{corollary}
\begin{proof}
    We play as Signgiver, answering to keep the discrepancy of the prefixes at most $d(n)$.
    If there are at least $3d+1$ unknown values, then both outcomes are still possible by Claim \ref{claim:discmax}.
\end{proof}

Therefore, if we could prove a constant bound on $d(n)$, that would also give a universal bound on $E_{n/2}(\textit{Disc-max-}d)$ for some constant $d$.
However, this is not possible.

\begin{claim}
    $\frac13\log n\le d(n)\le \frac32\sqrt n$.
\end{claim}
\begin{proof}
     To prove the upper bound, Signgiver may partition the index set $[n]$ into $\sqrt{n}$ intervals $I_1,I_2,\dots,I_{\sqrt{n}}$, each of length $\sqrt{n}$, and in each $I_h$ assign signs to the queries of Positioner in an alternating manner.
    Then the discrepancy over any complete $I_h$ is always 0 or 1, and within an interval $I_h$ the discrepancy is at most $\frac{1}{2}\sqrt{n}$.
    Thus, for any prefix $[1,j]$ at any time $t$, we have $disc_t(j)\le \sqrt{n}\cdot 1+ \frac{1}{2}\sqrt{n}=\frac32\sqrt{n}$.

    To prove the lower bound, we imagine that (the non-existent) $0$-th and $(n+1($-st positions have value $+1$ and $-1$, i.e., $x_0=+1$ and $x_{n+1}=-1$.
    First, we query $n/2$.
    If $x_{n/2}=-1$, next Positioner can query $x_{n/4}$, while if  $x_{n/2}=+1$, next Positioner queries $x_{3n/4}$.
    Positioner continues asking this way always the index between the nearest two opposite values.
    Positioner can do this for $\log n$ steps, after which the two nearest indices become adjacent.
    If at this stage, we have $\frac13\log n$ of $+1$ values in some prefix, then $d(n)\ge disc^t(j)\ge\frac13\log n$ for some $j$.
    Otherwise, we have at most $\frac13\log n$ of $+1$ values and at least $\frac23\log n$ of $-1$ values, so $d(n)\ge -disc^t(n)\ge\frac23\log n-\frac13\log n=\frac13\log n$.
\end{proof}

As this approach thus cannot give a constant bound on $E_{n/2}(\textit{Disc-max-}d)$, we prove a modified version of Claim \ref{claim:discmax} next, before which we introduce some new notation:\\

$disc^t(i,j)=\sum_{h=i}^j x^t_h$ \hspace{1cm} and \hspace{1cm}
$unq^t(i,j)=|\{i\le h\le j: x^t_h=0\}|$.\\

Note that instead of prefixes, now we consider intervals---this is because otherwise Claim \ref{claim:lyukas}/(i') would not hold.

\begin{claim}\label{claim:lyukas}\
    \begin{enumerate}
        \item[(i')] For any $d>0$, if $|disc^t(i,j)|\le d+unq^t(i,j)-3$ for every $i,j$ at some point $t$, then $\textit{Disc-max-}d(A)$ being true is possible.
        \item[(ii')] For any $d$, if $|disc^t(1,j)|\le d+\frac12 unq^t(1,j)$ for every $j$ at some point $t$, and there are yet $6d+1$ unqueried positions, then $\textit{Disc-max-}d(A)$ being false is possible.
    \end{enumerate}
\end{claim}
\begin{proof}
    (i')
    Let us first assume that $d$ is even and let $d'=d-2$. In this case actually assuming $|disc^t(i,j)|\le d+unq^t(i,j)-2=d'+unq^t(i,j)$ is already enough.

    Let us call an interval $(i,j)$ \emph{even} if $i$ is odd and $j$ is even. What we actually prove is that if $|disc^t(i,j)|\le d'+unq^t(i,j)$ for every even interval (such an even interval is called \emph{good}), then we can choose values for the unqueried positions in an arbitrary order such that every even interval remains good in every step. This way at the end, when there are no unqueried positions left, every even interval must have $|disc^t(i,j)|\le d'$ and so every interval has $|disc^t(i,j)|\le d'+2=d$, as claimed.

     Call an even interval $(i,j)$ \emph{strict} if $|disc^t(i,j)|= d'+unq^t(i,j)$. A strict even interval is \emph{positive}, if $disc^t(i,j)>0$ and is negative if $disc^t(i,j)<0$.
     Note that due to parity, an even interval is either strict or has $|disc^t(i,j)|\le  d'+unq^t(i,j)-2$ (this is where we use that $d$ is even and that the intervals have even length). Let $p$ be an unqueried position. Whatever value we choose for $p$, every non-strict even interval remains good.
     Consider the strict intervals that contain $p$. If all of them are positive or all of them are negative, then we can choose the opposite value for $p$ and they remain good.

     Assume now that $p$ is contained in a strict positive interval $I$ and a strict negative interval $J$ as well; we shall reach a contradiction. The intervals $I$ and $J$ are either intersecting or one contains the other; in both cases, the endpoints of $I$ and $J$ split $I\cup J$ into $3$ even intervals, some of them possibly empty: $K_1,K_2,K_3$ (this is where even length intervals are not enough, and we need even intervals), where $I\cap J=K_2$. Using that $I$ and $J$ are strict positive and negative, respectively, and that these $3$ intervals are good, we get the following (in)equalities:
     $$disc(I)=|disc(I)|= d'+unq(I)$$
     $$-disc(J)=|disc(J)|= d'+unq(J)$$
     $$|disc(K_1)|\le d'+unq(K_1)$$
     $$|disc(K_3)|\le d'+unq(K_3)$$
     $$disc(I)-disj(J)=\pm disc(K_1)\pm disc(K_3)\le |disc(K_1)|+|disc(K_3)|,$$
     where in the last line the sign of the $\pm's$ depends on how $I$ and $J$ intersect.\\
     The first two equations imply that $$disc(I)-disj(J)=2d'+unq(I)+unq(J)=2d'+unq(K_1)+unq(K_3)+2unq(K_2).$$
     On the other hand, the last three inequalities imply that
     $$disc(I)-disj(J)\le |disc(K_1)|+|disc(K_3)|\le 2d'+unq(K_1)+unq(K_3).$$
     These two can hold together only if $unq(K_2)=0$, yet $p$ is an unqueried position in $K_2$, a contradiction.

    The proof is finished for the case when $d$ is even. Recall that we just used (and maintained) that $|disc^t(i,j)|\le  d+unq^t(i,j)-2$ for every even interval, which is one more than what we actually assume in $(v)$.

    So let us assume now that $d$ is odd. In this case we can apply the even case for the even number $d-1$: the initial assumption holds, as
    $|disc^t(i,j)|\le  (d-1)+unq^t(i,j)-2$ and so we can maintain this in every step to get that at the end for every even interval $|disc^t(i,j)|\le (d-1)-2$ and so every interval has $|disc^t(i,j)|\le (d-1)-2+2=d-1<d$, finishing the case when $d$ is odd.

    (ii') Introduce $u=unq^t(1,n)$.
    So far we have assigned at most $d+\frac u2$ more $-1$ values than $+1$ values, as otherwise $disc^t(1,j)\le d+\frac u2$  would not hold, and vice versa.
    In particular, there are at least $\lceil\frac{u-(d+\frac u2)}2\rceil\ge d+1$ more $-1$'s and at least $d+1$ more $+1$'s we can assign.
    Let $j$ be the index of the $(d+1)$-st hitherto unqueried position.
    Without loss of generality, assume $disc^t(1,j)\ge 0$.
    If we assign $+1$'s to the first $d+1$ unqueried positions, then $|disc(1,j)|>d$, so $\textit{Disc-max-}d(x)$ is false.
\end{proof}

Now, we introduce slightly different two-player game to bound $E_{n/2}(\textit{Disc-max-}d)$.
As before, the board is an array $x$ of length $n$ that has only 0 entries at the beginning of the game, i.e., $x^0_i=0$ for all $i$. In the $t$-th round, a player called \textit{Positioner} asks an index $i$ with $x^{t-1}_i=0$ for some $1\le i\le n$, and then a player called \textit{Signgiver}, as an answer, sets $x^t_i=+1$ or $x^t_i=-1$.
The values of the other indices remain unchanged, i.e., $x^t_j=x^{t-1}_j$ for $j\ne i$.

In this new game, the goal of Positioner is to make $|disc^t(i,j)|-\frac12 unq^t(i,j))$ big over some interval $[i,j]$ at some time $t$, while the goal of Signgiver is to keep it low.
Denote the value of this game by $d'(n)$, that is, $\max_t\max_{1\le i\le j\le n}( |disc^t(i,j)|-\frac12 unq^t(i,j))$ during an optimal run of this game.

Obviously, $d'(n)\le 2d(n)$.
A simple corollary of Claim \ref{claim:lyukas} is the following.

\begin{corollary}
    For $d=d'(n)+3$, we have $E_{n/2}(\textit{Disc-max-}d)\le 6d$.
\end{corollary}
\begin{proof}
    We play as Signgiver, answering to keep $\max_t\max_{1\le i\le j\le n}( |disc^t(i,j)|-\frac12 unq^t(i,j))$ at most $d'(n)$.
    If there are at least $6d+1$ unknown values, then both outcomes are still possible by Claim \ref{claim:lyukas}.
\end{proof}

Therefore, if we could prove a constant bound on $d'(n)$, that would imply that $\textit{Disc-max-}d$ is an explicit function with a bounded $E_{n/2}$.
However, we could not prove such a bound, so we leave this open.

\begin{question}
How much is $d'(n)$?
\end{question}

We posed the more natural question of determining $d(n)$ on MathOverflow\footnote{\url{https://mathoverflow.net/q/446197/955}} but we did not receive any answers.

\bigskip
{\bf Acknowledgements.}
We would like to thank Farzan Byramji for discussions and Kristóf Zólomy for calling our attention to an error in an earlier version of our paper.

\smallskip
\textbf{Funding}: Research supported by the National Research, Development and Innovation Office - NKFIH under the grants FK 132060, PD 137779, KKP-133819, K 132696.

Research of B. Keszegh, D.T. Nagy and D. Pálvölgyi was also supported by the J\'anos Bolyai Research Scholarship of the Hungarian Academy of Sciences.

Research of B. Keszegh and D. Pálvölgyi was also supported by the ERC Advanced Grant ``ERMiD'' and by the New National Excellence Program \'UNKP-23-5 and by the Thematic Excellence Program TKP2021-NKTA-62 of the National Research, Development and Innovation Office.

Research of G. Wiener was supported by the  Research, Development and Innovation Fund, financed under the TKP2021 (project no. BME-NVA-02) funding scheme.



\end{document}